\documentclass{amsart}
\usepackage{amssymb}
\usepackage{amsfonts}

\newtheorem{theorem}{Theorem}
\theoremstyle{plain}

\newtheorem{example}{Example}

\newtheorem{lemma}{Lemma}

\newtheorem{remark}{Remark}

\numberwithin{equation}{section}
\input{tcilatex}

\begin{document}
\title[Weak and Strong Convergence]{Weak and Strong Convergence of an
Implicit Iteration for Hemi-Contractive Semigroups in Banach Spaces}
\author{FER\.{I}T G\"{U}RB\"{U}Z}
\address{Department of Mathematics, K\i rklareli University, K\i rklareli
39100, T\"{u}rkiye}
\email{feritgurbuz@klu.edu.tr }
\urladdr{}
\thanks{}
\curraddr{ }
\urladdr{}
\thanks{}
\date{}
\dedicatory{}
\thanks{}

\begin{abstract}
We study an implicit iteration for approximating common fixed points of
strongly continuous semigroups on closed convex subsets of real Banach
spaces. Under a stepwise Lipschitz restriction and hemi-contractivity
relative to the common fixed point set, we establish well-definedness and Fej%
\'{e}r monotonicity of the iteration. In reflexive Banach spaces satisfying
Opial's condition, we obtain a sufficient condition for weak convergence
based on a normalized residual estimate, uniform continuity at zero along
the generated sequence, and demiclosedness at fixed positive times. In
general real Banach spaces, we prove strong convergence under a
time-dependent directional contraction condition and a divergence condition
coupling the iteration weights with the contraction strength. A linear
example shows that vanishing time steps and a divergent sum of iteration
weights alone do not ensure convergence to a common fixed point. A nonlinear
planar evolution semigroup illustrates the applicability of the strong
convergence theorem beyond the nonexpansive class.
\end{abstract}

\maketitle

\section{Introduction}

Iterative approximation of fixed points is a central topic in nonlinear
analysis. Its development includes the mean-value iteration of Mann \cite%
{Mann} and numerous subsequent methods whose convergence depends on the
interaction between operator properties, the geometry of the underlying
space, and the iteration parameters. For background on Banach space geometry
and nonlinear iteration methods, we refer to \cite{Chidume}.

For a strongly continuous semigroup of nonlinear mappings, the approximation
problem concerns a common fixed point of the entire family rather than a
fixed point of a single selected mapping. Common fixed points represent
stationary states of the evolution described by the semigroup. Iterative
approximation in this setting therefore involves both iteration weights and
semigroup times.

Implicit approximation methods have an established history. Xu and Ori \cite%
{Xu} studied an implicit iteration for finite families of nonexpansive
mappings. In the semigroup setting, Suzuki \cite{Suzuki} investigated strong
convergence of Krasnoselskii and Mann type sequences, while Thong \cite%
{Thong} studied an implicit iteration for nonexpansive semigroups. Related
developments beyond nonexpansiveness include the explicit iteration studied
by Chang et al.\cite{Chang} for Lipschitzian demicontraction semigroups and
the implicit process considered by Quan et al. \cite{Quan} for
pseudocontractive semigroups. Kim \cite{Kim} established weak convergence of
an implicit process for Lipschitzian pseudocontraction semigroups in
uniformly convex Banach spaces satisfying Opial's condition.

Ofem et al. \cite{Ofem} also investigated implicit iterations for
Lipschitzian $\alpha $-hemicontraction semigroups. Their parameter scales
the fixed-point component in the defining inequality. The time-dependent
directional contraction condition used in the present paper has a different
role: its coefficient measures the contraction strength relative to a common
fixed point and is allowed to approach zero with the semigroup time. We
therefore do not use the term $\alpha $-hemicontractive for this condition.

Kozlowski \cite{Kozlowski} studied the same recursive form of implicit
iteration for asymptotic nonexpansive semigroups. That work uses a
normalized residual condition at vanishing times and a telescoping argument
to obtain residual estimates at fixed positive times. Our residual-transfer
lemma follows this established approach, with locally uniform Lipschitz
bounds and continuity assumptions imposed along the sequence under
consideration. The weak convergence result below assumes demiclosedness
explicitly; it does not derive this property from hemi-contractivity.
Accordingly, we do not claim that it subsumes the convergence theorems cited
above.

In this paper, we consider the iteration%
\begin{equation*}
x_{n+1}=\left( 1-a_{n}\right) x_{n}+a_{n}T\left( t_{n}\right) x_{n+1},\qquad
n\geq 1,
\end{equation*}%
on a nonempty closed convex subset of a real Banach space. A stepwise
Lipschitz restriction ensures that each implicit equation has a unique
solution. Hemi-contractivity relative to the common fixed point set then
yields Fej\'{e}r monotonicity.

We establish two convergence results under distinct additional assumptions.
The first is a sufficient condition for weak convergence in a reflexive
Banach space satisfying Opial's condition \cite{Opial}. It requires a
normalized residual estimate, uniform continuity at zero along the generated
sequence, and demiclosedness at fixed positive times. These are assumptions
to be verified separately, rather than conclusions obtained solely from the
iteration controls. The second result gives strong convergence in a general
real Banach space under a time-dependent directional contraction condition.
Its divergence condition couples the iteration weights with the directional
contraction strength.

A linear example demonstrates why vanishing semigroup times and a divergent
sum of iteration weights are insufficient on their own. We also construct a
nonlinear planar evolution semigroup satisfying the strong convergence
assumptions while failing to be nonexpansive. This application illustrates
the distinction between contraction relative to a common equilibrium and
nonexpansiveness between arbitrary pairs of points.

The paper is organized as follows. Section 2 presents the definitions and
auxiliary lemmas. Section 3 establishes the convergence results and examines
the time--weight interaction in a linear example. Section 4 gives the
nonlinear evolution application. Section 5 discusses the scope and
limitations of the results, and Section 6 concludes the paper.

\section{Preliminaries and Auxiliary Results}

Throughout the sequel, $E$ denotes a real Banach space and $E^{\ast }$ its
dual space. All mappings under consideration act on a fixed nonempty closed
and convex subset $C\subset E$. The duality pairing between $E$ and $E^{\ast
}$ is denoted by $\left \langle \cdot ,\cdot \right \rangle $. For a
sequence $\left \{ x_{n}\right \} \subset E$, the notations $%
x_{n}\rightarrow x$ and $x_{n}\rightharpoonup x$ stand for strong and weak
convergence, respectively.

\subsection{Duality mapping and geometric properties}

The normalized duality mapping $J:E\rightarrow 2^{E^{\ast }}$ is defined by 
\begin{equation*}
J\left( x\right) =\left \{ f\in E^{\ast }:\left \langle x,f\right \rangle
=\left \Vert x\right \Vert ^{2},\text{ }\left \Vert f\right \Vert =\left
\Vert x\right \Vert \right \} ,\qquad \text{ }x\in E.
\end{equation*}

The duality mapping plays a central role in the study of nonlinear operators
in Banach spaces, particularly in the structural analysis of
pseudocontractive and related generalized contractive mappings (see \cite%
{Chidume, Petryshyn}).

A Banach space $E$ is said to satisfy Opial's condition if, whenever a
sequence $\left( x_{n}\right) $ converges weakly to $x\in E$, one has

\begin{equation*}
\liminf \limits_{n\rightarrow \infty }\left \Vert x_{n}-x\right \Vert
<\liminf \limits_{n\rightarrow \infty }\left \Vert x_{n}-y\right \Vert
\end{equation*}%
for every $y\in E\diagdown \left \{ x\right \} $. This condition was
introduced by Opial \cite{Opial}. Real Hilbert spaces and the sequence
spaces $l^{p}$, $1<p<\infty $, satisfy Opial's condition. Uniform convexity
alone, however, does not imply this property; therefore, Opial's condition
will be stated explicitly whenever it is required. For background on the
geometry of Banach spaces and nonlinear iteration methods, see \cite{Chidume}%
.

\subsection{Hemi-contractive mappings}

Let $C$ be a nonempty subset of a real Banach space $E$, and let $%
T:C\rightarrow C$ be a mapping with nonempty fixed point set 
\begin{equation*}
\text{Fix}\left( T\right) :=\left \{ x\in C:Tx=x\right \} .
\end{equation*}

The mapping $T$ is called hemi-contractive if, for every $x\in C$ and every $%
p\in $ Fix$\left( T\right) $, there exists $j_{x,p}\in J\left( x-p\right) $
such that

\begin{equation}
\left \langle Tx-p,j_{x,p}\right \rangle \leq \left \Vert x-p\right \Vert
^{2}.  \label{2.1}
\end{equation}%
For background on hemi-contractive mappings and their role in nonlinear
iteration theory, see \cite{Chidume}.

Equivalently, the same choice of $j_{x,p}$ satisfies%
\begin{equation*}
\left \langle x-Tx,j_{x,p}\right \rangle \geq 0.
\end{equation*}%
The element $j_{x,p}$ may depend on both $x$ and $p$. No single-valuedness
or weak continuity assumption is imposed on the normalized duality mapping $%
J $.

In a real Hilbert space, the normalized duality mapping is identified with
the identity, and condition (\ref{2.1}) becomes%
\begin{equation*}
\left \langle Tx-p,x-p\right \rangle \leq \left \Vert x-p\right \Vert
^{2}\qquad x\in C,\text{ }p\in \text{Fix}\left( T\right) .
\end{equation*}

\subsection{Strongly continuous semigroups}

Let $C$ be a nonempty closed convex subset of a real Banach space $E$. A
family 
\begin{equation*}
\mathcal{T}=\left \{ T\left( t\right) :t\geq 0\right \}
\end{equation*}%
of mappings from $C$ into itself is called a strongly continuous semigroup
if the following conditions hold:

$\left( i\right) $ $T\left( 0\right) x=x$ for every $x\in C;$

$\left( ii\right) $ $T\left( s+t\right) x=T\left( s\right) T\left( t\right) x
$ for every $s,t\geq 0$ and every $x\in C;$

$\left( iii\right) $ for every $x\in C$, the mapping $t\mapsto T\left(
t\right) x$ is norm-continuous on $\left[ 0,\infty \right) $.

The common fixed point set of $\mathcal{T}$ is denoted by%
\begin{equation*}
F\left( \mathcal{T}\right) :=\bigcap \limits_{t\geq 0}\text{Fix}\left(
T\left( t\right) \right) =\left \{ p\in C:T\left( t\right) p=p\text{ for
every }t\geq 0\right \} .
\end{equation*}%
Throughout the convergence analysis, we assume that $F\left( \mathcal{T}%
\right) \neq \emptyset $. We assume that, for each $t>0$, there exists a
finite constant $L\left( t\right) \geq 0$ such that%
\begin{equation*}
\left \Vert T\left( t\right) x-T\left( t\right) y\right \Vert \leq L\left(
t\right) \left \Vert x-y\right \Vert ,\qquad x,y\in C.
\end{equation*}%
The Lipschitz bound is allowed to depend on $t$. A uniform bound on $L\left(
t\right) $ over all positive times is not assumed.

The family $\mathcal{T}$ is said to be hemi-contractive with respect to its
common fixed point set if, for every $t>0$, every $x\in C$, and every $p\in
F\left( \mathcal{T}\right) $, there exists $j_{t,x,p}\in J\left( x-p\right) $
such that%
\begin{equation*}
\left \langle T\left( t\right) x-p,j_{t,x,p}\right \rangle \leq \left \Vert
x-p\right \Vert ^{2}.
\end{equation*}%
The choice of $j_{t,x,p}$ may depend on $t$, $x$, and $p$. In particular,
this condition holds whenever each $T\left( t\right) $, $t\geq 0$, is
hemi-contractive in the sense of Section 2.2. For related studies on fixed
point approximation for nonlinear semigroups, see \cite{Kim, Thong}.

\subsection{Auxiliary Lemmas}

The following lemmas provide the basic tools for the convergence analysis.
Lemma 1 concerns a scalar recursive inequality, while Lemma 2 gives a weak
convergence criterion based on Opial's condition. Lemma 3 establishes the
well-definedness and Fej\'{e}r monotonicity of the implicit iteration. Lemma
4 provides sufficient conditions for transferring residual estimates from
vanishing time steps to fixed positive times. These results distinguish the
properties that follow directly from hemi-contractivity from those that
require additional assumptions on the semigroup and the time parameters.

\begin{lemma}
\label{Lemma1} Let $\left( u_{n}\right) $ be a sequence of nonnegative real
numbers satisfying 
\begin{equation*}
u_{n+1}\leq \left( 1-\gamma _{n}\right) u_{n}+\varepsilon _{n},\qquad n\geq
1,
\end{equation*}%
where $\left \{ \gamma _{n}\right \} \subset \left( 0,1\right) $ and $\left(
\varepsilon _{n}\right) \subset \left[ 0,\infty \right) $. If%
\begin{equation*}
\sum \limits_{n=1}^{\infty }\gamma _{n}=\infty \qquad \text{and}\qquad \sum
\limits_{n=1}^{\infty }\varepsilon _{n}<\infty ,
\end{equation*}%
then $u_{n}\rightarrow 0$ as $n\rightarrow \infty $.
\end{lemma}

\begin{proof}
Fix an integer $N\geq 1$. Repeated application of the recursive inequality
gives, for every $n\geq N$,%
\begin{equation*}
u_{n+1}\leq u_{N}\prod \limits_{k=N}^{n}\left( 1-\gamma _{k}\right) +\sum
\limits_{k=N}^{n}\varepsilon _{k}\prod \limits_{j=k+1}^{n}\left( 1-\gamma
_{j}\right) ,
\end{equation*}%
where an empty product is interpreted as $1$. Since every factor belongs to $%
\left( 0,1\right) $, we obtain%
\begin{equation*}
u_{n+1}\leq u_{N}\prod \limits_{k=N}^{n}\left( 1-\gamma _{k}\right) +\sum
\limits_{k=N}^{n}\varepsilon _{k}.
\end{equation*}%
Moreover,%
\begin{equation*}
0\leq \prod \limits_{k=N}^{n}\left( 1-\gamma _{k}\right) \leq \exp \left(
-\sum \limits_{k=N}^{n}\gamma _{k}\right) \rightarrow 0.
\end{equation*}%
Consequently,%
\begin{equation*}
0\leq \limsup \limits_{n\rightarrow \infty }u_{n+1}\leq \sum
\limits_{k=N}^{\infty }\varepsilon _{k}.
\end{equation*}%
Letting $N\rightarrow \infty $ and using the summability of $\left(
\varepsilon _{n}\right) $, we conclude that $u_{n}\rightarrow 0$.
\end{proof}

The following lemma is a standard consequence of Opial's condition. For
background on Opial's condition and weak convergence methods in Banach
spaces, see \cite{Chidume, Opial}. We include a proof for completeness.

\begin{lemma}
\label{Lemma2} Let $E$ be a reflexive Banach space satisfying Opial's
condition, and let $D\subset E$ be nonempty. Suppose that $\left(
x_{n}\right) $ is a bounded sequence in $E$ such that:

$\left( i\right) $ for every $p\in D$, the limit%
\begin{equation*}
\lim \limits_{n\rightarrow \infty }\left \Vert x_{n}-p\right \Vert
\end{equation*}%
exists;

$\left( ii\right) $ every weak subsequential limit of $\left( x_{n}\right) $
belongs to $D$. Then $\left( x_{n}\right) $ converges weakly to a point in $%
D $.
\end{lemma}

\begin{proof}
Since $E$ is reflexive and $\left( x_{n}\right) $ is bounded, there exists a
subsequence $\left( x_{n_{k}}\right) $ converging weakly to some $u\in E$.
By $\left( ii\right) $, $u\in D$.

We first show that $u$ is the unique weak subsequential limit. Suppose that
another subsequence $\left( x_{m_{k}}\right) $ converges weakly to $v\neq u$%
. Again, $\left( ii\right) $ gives $v\in D$. By $\left( i\right) $, both
limits%
\begin{equation*}
l_{u}:=\lim \limits_{n\rightarrow \infty }\left \Vert x_{n}-u\right \Vert
,\qquad l_{v}:=\lim \limits_{n\rightarrow \infty }\left \Vert x_{n}-v\right
\Vert
\end{equation*}%
exist. Applying Opial's condition to $\left( x_{n_{k}}\right) $, we obtain%
\begin{equation*}
l_{u}=\liminf \limits_{k\rightarrow \infty }\left \Vert x_{n_{k}}-u\right
\Vert <\liminf \limits_{k\rightarrow \infty }\left \Vert x_{n_{k}}-v\right
\Vert =l_{v}.
\end{equation*}%
Applying the same condition to $\left( x_{m_{k}}\right) $ yields $%
l_{v}<l_{u} $, a contradiction. Thus $u$ is the unique weak subsequential
limit.

To prove weak convergence of the entire sequence, suppose that $x_{n}$ does
not converge weakly to $u$. Then there exist $f\in E^{\ast }$, $\delta >0$,
and a subsequence $\left( x_{r_{k}}\right) $ such that%
\begin{equation*}
\left \vert f\left( x_{r_{k}}-u\right) \right \vert \geq \delta 
\end{equation*}%
for every $k$. Boundedness and reflexivity ensure that $\left(
x_{r_{k}}\right) $ has a further subsequence converging weakly to some $w\in
E$. By uniqueness, $w=u$, contradicting the displayed inequality. Therefore, 
$x_{n}\rightharpoonup u\in D$.
\end{proof}

\begin{lemma}
\label{Lemma3}\textbf{(Well-definedness and Fej\'{e}r monotonicity)} Let $C$
be a nonempty closed convex subset of a real Banach space $E$, and let $%
\mathcal{T}=\left \{ T\left( t\right) :t\geq 0\right \} $ be a strongly
continuous semigroup on $C$. Assume that $F\left( \mathcal{T}\right) \neq
\emptyset $, that $\mathcal{T}$ is hemi-contractive with respect to $F\left( 
\mathcal{T}\right) $, and that, for each $t>0$,%
\begin{equation*}
\left \Vert T\left( t\right) x-T\left( t\right) y\right \Vert \leq L\left(
t\right) \left \Vert x-y\right \Vert ,\qquad x,y\in C,
\end{equation*}%
where $L\left( t\right) $ is finite.

Let $x_{1}\in C$, $\left( a_{n}\right) \subset \left( 0,1\right) $, and $%
\left( t_{n}\right) \subset \left( 0,\infty \right) $. Suppose that%
\begin{equation*}
a_{n}L\left( t_{n}\right) <1
\end{equation*}%
for every $n\geq 1$. Then the implicit iteration%
\begin{equation*}
x_{n+1}=\left( 1-a_{n}\right) x_{n}+a_{n}T\left( t_{n}\right) x_{n+1},\qquad
n\geq 1,
\end{equation*}%
defines a unique sequence $\left( x_{n}\right) \subset C$. Moreover,%
\begin{equation*}
\left \Vert x_{n+1}-p\right \Vert \leq \left \Vert x_{n}-p\right \Vert
,\qquad p\in F\left( \mathcal{T}\right) \text{, }n\geq 1.
\end{equation*}%
Consequently, $\left( x_{n}\right) $ is bounded, and the limit%
\begin{equation*}
\lim \limits_{n\rightarrow \infty }\left \Vert x_{n}-p\right \Vert
\end{equation*}%
exists for every $p\in F\left( \mathcal{T}\right) $.
\end{lemma}

\begin{proof}
Suppose that $x_{n}\in C$ has been determined. Define $S_{n}:C\rightarrow C$
by%
\begin{equation*}
S_{n}\left( y\right) :=\left( 1-a_{n}\right) x_{n}+a_{n}T\left( t_{n}\right)
y.
\end{equation*}%
Since $C$ is convex and $T\left( t_{n}\right) $ maps $C$ into itself, $%
S_{n}\left( C\right) \subset C$. Furthermore,%
\begin{equation*}
\left \Vert S_{n}\left( y\right) -S_{n}\left( z\right) \right \Vert \leq
a_{n}L\left( t_{n}\right) \left \Vert y-z\right \Vert ,\qquad y,z\in C.
\end{equation*}%
Thus $S_{n}$ is a contraction. Since $C$ is a closed subset of the Banach
space $E$, it is complete. Banach's contraction principle therefore yields a
unique fixed point $x_{n+1}\in C$ of $S_{n}$. Applying this argument
successively establishes the existence and uniqueness of the sequence.

Fix $p\in F\left( \mathcal{T}\right) $. If $x_{n+1}=p$, the asserted
inequality is immediate. Otherwise, by hemi-contractivity with respect to $%
F\left( \mathcal{T}\right) $, choose%
\begin{equation*}
j\in J\left( x_{n+1}-p\right)
\end{equation*}%
such that%
\begin{equation*}
\left \langle T\left( t_{n}\right) x_{n+1}-p,j\right \rangle \leq \left
\Vert x_{n+1}-p\right \Vert ^{2}.
\end{equation*}%
Pairing the iteration equation with $j$, we obtain%
\begin{eqnarray*}
\left \Vert x_{n+1}-p\right \Vert ^{2} &=&\left( 1-a_{n}\right) \left
\langle x_{n}-p,j\right \rangle +a_{n}\left \langle T\left( t_{n}\right)
x_{n+1}-p,j\right \rangle \\
&\leq &\left( 1-a_{n}\right) \left \Vert x_{n}-p\right \Vert \left \Vert
x_{n+1}-p\right \Vert +a_{n}\left \Vert x_{n+1}-p\right \Vert ^{2}.
\end{eqnarray*}%
Here we used $\left \Vert j\right \Vert =\left \Vert x_{n+1}-p\right \Vert $%
. Rearranging gives%
\begin{equation*}
\left( 1-a_{n}\right) \left \Vert x_{n+1}-p\right \Vert ^{2}\leq \left(
1-a_{n}\right) \left \Vert x_{n}-p\right \Vert \left \Vert x_{n+1}-p\right
\Vert .
\end{equation*}%
Dividing by the positive quantity%
\begin{equation*}
\left( 1-a_{n}\right) \left \Vert x_{n+1}-p\right \Vert
\end{equation*}%
yields%
\begin{equation*}
\left \Vert x_{n+1}-p\right \Vert \leq \left \Vert x_{n}-p\right \Vert .
\end{equation*}%
Hence, for each $p\in F\left( \mathcal{T}\right) $, the sequence $\left(
\left \Vert x_{n}-p\right \Vert \right) $ is nonincreasing and bounded below
by zero, so it converges. Since $F\left( \mathcal{T}\right) \neq \emptyset $%
, the same inequality also establishes boundedness of $\left( x_{n}\right) $.
\end{proof}

The next lemma uses the telescoping argument appearing in the proof of Lemma
4.3 of Kozlowski \cite{Kozlowski}. We state it under locally uniform
Lipschitz bounds and a continuity assumption along the given sequence, and
include the proof for completeness.

\begin{lemma}
\label{Lemma4}\textbf{(Transfer of residuals to fixed positive times) }Let $%
\mathcal{T}=\left \{ T\left( t\right) :t\geq 0\right \} $ be a strongly
continuous semigroup on a subset $C$ of a Banach space $E$. Suppose that,
for every $S>0$, there exists $M_{S}\geq 1$ such that%
\begin{equation*}
\left \Vert T\left( t\right) x-T\left( t\right) y\right \Vert \leq
M_{S}\left \Vert x-y\right \Vert ,\qquad x,y\in C,\text{ }0\leq t\leq S.
\end{equation*}%
Let $\left( z_{n}\right) \subset C$ and let $\left( t_{n}\right) \subset
\left( 0,\infty \right) $ satisfy $t_{n}\rightarrow 0$. Assume that%
\begin{equation}
\lim \limits_{\delta \downarrow 0}\sup_{n\geq 1}\sup_{0\leq r\leq \delta
}\left \Vert T\left( r\right) z_{n}-z_{n}\right \Vert =0  \label{2.4}
\end{equation}%
and%
\begin{equation*}
\frac{\left \Vert T\left( t_{n}\right) z_{n}-z_{n}\right \Vert }{t_{n}}%
\rightarrow 0.
\end{equation*}%
Then, for every fixed $s>0$,%
\begin{equation*}
\left \Vert T\left( s\right) z_{n}-z_{n}\right \Vert \rightarrow 0.
\end{equation*}
\end{lemma}

\begin{proof}
Fix $s>0$. For all sufficiently large $n,t_{n}<s$. Define%
\begin{equation*}
m_{n}:=\left \lfloor \frac{s}{t_{n}}\right \rfloor ,\qquad
r_{n}:=s-m_{n}t_{n}.
\end{equation*}%
Then $0\leq r_{n}<t_{n}$ and $m_{n}t_{n}\leq s$. By the semigroup property
and the Lipschitz bound,%
\begin{eqnarray*}
\left \Vert T\left( s\right) z_{n}-T\left( m_{n}t_{n}\right) z_{n}\right
\Vert &=&\left \Vert T\left( m_{n}t_{n}\right) T\left( r_{n}\right)
z_{n}-T\left( m_{n}t_{n}\right) z_{n}\right \Vert \\
&\leq &M_{s}\left \Vert T\left( r_{n}\right) z_{n}-z_{n}\right \Vert .
\end{eqnarray*}%
Furthermore, a telescoping sum gives%
\begin{eqnarray*}
\left \Vert T\left( m_{n}t_{n}\right) z_{n}-z_{n}\right \Vert &\leq &\sum
\limits_{k=0}^{m_{n}-1}\left \Vert T\left( \left( k+1\right) t_{n}\right)
z_{n}-T\left( kt_{n}\right) z_{n}\right \Vert \\
&=&\sum \limits_{k=0}^{m_{n}-1}\left \Vert T\left( kt_{n}\right) T\left(
t_{n}\right) z_{n}-T\left( kt_{n}\right) z_{n}\right \Vert \\
&\leq &m_{n}M_{s}\left \Vert T\left( t_{n}\right) z_{n}-z_{n}\right \Vert \\
&\leq &sM_{s}\frac{\left \Vert T\left( t_{n}\right) z_{n}-z_{n}\right \Vert 
}{t_{n}}.
\end{eqnarray*}%
Combining these estimates, we obtain%
\begin{equation*}
\left \Vert T\left( s\right) z_{n}-z_{n}\right \Vert \leq M_{s}\left \Vert
T\left( r_{n}\right) z_{n}-z_{n}\right \Vert +sM_{s}\frac{\left \Vert
T\left( t_{n}\right) z_{n}-z_{n}\right \Vert }{t_{n}}.
\end{equation*}%
The first term tends to zero by (\ref{2.4}), since $0\leq
r_{n}<t_{n}\rightarrow 0$. The second term tends to zero by the normalized
residual assumption. This proves the assertion.
\end{proof}

\subsection{Demiclosedness at zero}

Let $C$ be a nonempty closed convex subset of a real Banach space $E$, and
let $S:C\rightarrow C$. The mapping $I-S$ is said to be demiclosed at zero
if, whenever a sequence $\left( z_{n}\right) \subset C$ satisfies%
\begin{equation*}
z_{n}\rightharpoonup z\in C\qquad \text{and}\qquad \left \Vert
z_{n}-Sz_{n}\right \Vert \rightarrow 0,
\end{equation*}%
one has $Sz=z$.

For a semigroup $\mathcal{T}=\left \{ T\left( t\right) :t\geq 0\right \} $,
the demiclosedness condition used below is that $I-T\left( s\right) $ is
demiclosed at zero for every fixed $s>0$. This condition is imposed
explicitly; it is not inferred from the hemi-contractive assumption. In
particular, suppose that $z_{n}\rightharpoonup z\in C$ and%
\begin{equation*}
\left \Vert z_{n}-T\left( s\right) z_{n}\right \Vert \rightarrow 0
\end{equation*}%
for every fixed $s>0$.

If $I-T\left( s\right) $ is demiclosed at zero for every $s>0$, then $%
T\left( s\right) z=z$ for every $s>0$, and hence $z\in F\left( \mathcal{T}%
\right) $.

\section{Main Results}

We consider an implicit iteration for approximating common fixed points of a
strongly continuous semigroup. The well-definedness and Fej\'{e}r
monotonicity of the iteration follow from Lemma \ref{Lemma3}. Additional
assumptions are needed to establish convergence to a common fixed point. In
particular, the weak convergence analysis requires residual estimates at
fixed positive times and a demiclosedness condition. These assumptions will
be stated explicitly in the corresponding theorem.

\subsection{The implicit iteration}

Let $C$ be a nonempty closed convex subset of a real Banach space $E$, and
let $\mathcal{T}=\left \{ T\left( t\right) :t\geq 0\right \} $ be a strongly
continuous semigroup on $C$ with $F\left( \mathcal{T}\right) \neq \emptyset $%
. Assume that $\mathcal{T}$ is hemi-contractive with respect to $F\left( 
\mathcal{T}\right) $. For each $t>0$, let $L\left( t\right) $ be a finite
Lipschitz bound for $T\left( t\right) $.

Given $x_{1}\in C$, let $\left( a_{n}\right) \subset \left( 0,1\right) $ and 
$\left( t_{n}\right) \subset \left( 0,\infty \right) $ satisfy 
\begin{equation*}
a_{n}L\left( t_{n}\right) <1
\end{equation*}%
for every $n\geq 1$.

Define $\left( x_{n}\right) $ recursively by 
\begin{equation}
x_{n+1}=\left( 1-a_{n}\right) x_{n}+a_{n}T\left( t_{n}\right) x_{n+1},\qquad
n\geq 1.  \label{1}
\end{equation}%
By Lemma \ref{Lemma3}, each implicit equation has a unique solution in $C$.
Moreover, $\left( x_{n}\right) $ is Fej\'{e}r monotone with respect to $%
F\left( \mathcal{T}\right) $; that is,%
\begin{equation*}
\left \Vert x_{n+1}-p\right \Vert \leq \left \Vert x_{n}-p\right \Vert
,\qquad p\in F\left( \mathcal{T}\right) ,\text{ }n\geq 1.
\end{equation*}

Consequently, $\left( x_{n}\right) $ is bounded, and $\lim
\limits_{n\rightarrow \infty }\left \Vert x_{n}-p\right \Vert $ exists for
every $p\in F\left( \mathcal{T}\right) $. For later use, rearranging (\ref{1}%
) gives%
\begin{equation*}
x_{n+1}-x_{n}=\frac{a_{n}}{1-a_{n}}\left( T\left( t_{n}\right)
x_{n+1}-x_{n+1}\right) .
\end{equation*}%
Hence,%
\begin{equation}
\left \Vert x_{n+1}-x_{n}\right \Vert =\frac{a_{n}}{1-a_{n}}\left \Vert
T\left( t_{n}\right) x_{n+1}-x_{n+1}\right \Vert .  \label{2}
\end{equation}

\subsection{A sufficient condition for weak convergence}

The following theorem combines Fej\'{e}r monotonicity with a normalized
residual condition and demiclosedness at fixed positive times.

\begin{theorem}
\label{Theorem 1} Let $E$ be a reflexive real Banach space satisfying
Opial's condition, and let $C\subset E$ be nonempty, closed, and convex. Let 
$\mathcal{T}=\left \{ T\left( t\right) :t\geq 0\right \} $ be a strongly
continuous semigroup on $C$ with $F\left( \mathcal{T}\right) \neq \emptyset $%
. Assume that $\mathcal{T}$ is hemi-contractive with respect to $F\left( 
\mathcal{T}\right) $ and that, for every $S>0$, there exists $M_{S}\geq 1$
such that%
\begin{equation*}
\left \Vert T\left( t\right) x-T\left( t\right) y\right \Vert \leq
M_{S}\left \Vert x-y\right \Vert ,\qquad x,y\in C,\text{ }0\leq t\leq S.
\end{equation*}%
Suppose also that $I-T\left( s\right) $ is demiclosed at zero for every $s>0$%
.

Let $\left( x_{n}\right) $ be generated by (\ref{1}), where%
\begin{equation*}
a_{n}\in \left( 0,1\right) ,\qquad t_{n}>0,\qquad a_{n}L\left( t_{n}\right)
<1\text{ }\left( n\geq 1\right) ,
\end{equation*}%
and $t_{n}\rightarrow 0$. Assume that%
\begin{equation}
\lim \limits_{\delta \downarrow 0}\sup_{n\geq 1}\sup_{0\leq r\leq \delta
}\left \Vert T\left( r\right) x_{n}-x_{n}\right \Vert =0  \label{3}
\end{equation}%
and 
\begin{equation}
\frac{\left \Vert T\left( t_{n}\right) x_{n}-x_{n}\right \Vert }{t_{n}}%
\rightarrow 0.  \label{4}
\end{equation}%
Then $\left( x_{n}\right) $ converges weakly to a point in $F\left( \mathcal{%
T}\right) $.
\end{theorem}

\begin{proof}
By Lemma \ref{Lemma3}, the sequence is well-defined and bounded. Moreover,
for every $p\in F\left( \mathcal{T}\right) $, the limit%
\begin{equation*}
\lim \limits_{n\rightarrow \infty }\left \Vert x_{n}-p\right \Vert
\end{equation*}%
exists. The local uniform Lipschitz bound, together with (\ref{3}) and (\ref%
{4}), allows us to apply Lemma \ref{Lemma4} with $z_{n}=x_{n}$. Thus,%
\begin{equation*}
\left \Vert T\left( s\right) x_{n}-x_{n}\right \Vert \rightarrow 0
\end{equation*}%
for every fixed $s>0$.

Let $z$ be any weak subsequential limit of $\left( x_{n}\right) $, and
choose a subsequence such that $x_{n_{k}}\rightharpoonup z$. Since $C$ is
closed and convex, it is weakly closed, so $z\in C$. For each fixed $s>0$,%
\begin{equation*}
\left \Vert x_{n_{k}}-T\left( s\right) x_{n_{k}}\right \Vert \rightarrow 0.
\end{equation*}%
The demiclosedness of $I-T\left( s\right) $ at zero yields $T\left( s\right)
z=z$. Since $s>0$ was arbitrary and $T\left( 0\right) =I$, we conclude that $%
z\in F\left( \mathcal{T}\right) $.

Therefore, every weak subsequential limit belongs to $F\left( \mathcal{T}%
\right) $. Applying Lemma \ref{Lemma2} with $D=F\left( \mathcal{T}\right) $,
we obtain weak convergence of the entire sequence to a point in $F\left( 
\mathcal{T}\right) $.
\end{proof}

\begin{remark}
Conditions (\ref{3}) and (\ref{4}) are additional assumptions on the
generated sequence. They are not consequences of strong continuity and Fej%
\'{e}r monotonicity alone. In particular, when $t_{n}\rightarrow 0$, the
unnormalized condition%
\begin{equation*}
\left \Vert T\left( t_{n}\right) x_{n}-x_{n}\right \Vert \rightarrow 0
\end{equation*}%
does not by itself imply that weak subsequential limits are common fixed
points.
\end{remark}

\subsection{Strong convergence under a time-dependent directional
contraction condition}

We next establish strong convergence under a quantitative directional
contraction condition. Unlike a strict coefficient bounded away from one
uniformly in time, the condition below allows the contraction strength to
vanish as the semigroup time approaches zero.

\begin{theorem}
\label{Theorem 2}Let $C$ be a nonempty closed convex subset of a real Banach
space $E$, and let $\mathcal{T}=\left \{ T\left( t\right) :t\geq 0\right \} $
be a strongly continuous semigroup on $C$ with $F\left( \mathcal{T}\right)
\neq \emptyset $. Suppose that, for each $t>0$, $T\left( t\right) $ has a
finite Lipschitz bound $L\left( t\right) $. Assume that there exists a
function%
\begin{equation*}
\eta :\left( 0,\infty \right) \rightarrow \left( 0,1\right] 
\end{equation*}%
such that, for every $t>0$, $x\in C$, and $p\in F\left( \mathcal{T}\right) $%
, there exists $j_{t,x,p}\in J\left( x-p\right) $ satisfying%
\begin{equation}
\left \langle T\left( t\right) x-p,j_{t,x,p}\right \rangle \leq \left( 1-\eta
\left( t\right) \right) \left \Vert x-p\right \Vert ^{2}.  \label{5}
\end{equation}%
Given $x_{1}\in C$, let $\left( x_{n}\right) $ be generated by (\ref{1}),
where%
\begin{equation*}
a_{n}\in \left( 0,1\right) ,\qquad t_{n}>0,\qquad a_{n}L\left( t_{n}\right)
<1\text{ }\left( n\geq 1\right) .
\end{equation*}%
If%
\begin{equation}
\sum \limits_{n=1}^{\infty }a_{n}\eta \left( t_{n}\right) =\infty ,  \label{6}
\end{equation}%
then $F\left( \mathcal{T}\right) $ consists of a single point $p$, and $%
x_{n}\rightarrow p$ in norm.
\end{theorem}

\begin{proof}
Condition (\ref{5}) implies hemi-contractivity with respect to $F\left( 
\mathcal{T}\right) $. Hence, by Lemma \ref{Lemma3}, the sequence is
well-defined.

We first establish uniqueness of the common fixed point. Let $p,q\in F\left( 
\mathcal{T}\right) $ and fix $t>0$. Applying (\ref{5}) with $x=q$, and using 
$T\left( t\right) q=q$, gives%
\begin{equation*}
\left \Vert q-p\right \Vert ^{2}\leq \left( 1-\eta \left( t\right) \right)
\left \Vert q-p\right \Vert ^{2}.
\end{equation*}%
Since $\eta \left( t\right) >0$, we obtain $q=p$. Thus $F\left( \mathcal{T}%
\right) =\left \{ p\right \} $.

If $x_{n+1}\neq p$, choose $j\in J\left( x_{n+1}-p\right) $ satisfying (\ref%
{5}) at time $t_{n}$. Pairing the iteration equation with $j$, we obtain%
\begin{eqnarray*}
\left \Vert x_{n+1}-p\right \Vert ^{2} &=&\left( 1-a_{n}\right) \left
\langle x_{n}-p,j\right \rangle +a_{n}\left \langle T\left( t_{n}\right)
x_{n+1}-p,j\right \rangle \\
&\leq &\left( 1-a_{n}\right) \left \Vert x_{n}-p\right \Vert \left \Vert
x_{n+1}-p\right \Vert +a_{n}\left( 1-\eta \left( t_{n}\right) \right) \left
\Vert x_{n+1}-p\right \Vert ^{2}.
\end{eqnarray*}%
Rearranging and dividing by $\left \Vert x_{n+1}-p\right \Vert $ yields%
\begin{equation}
\left \Vert x_{n+1}-p\right \Vert \leq \frac{1-a_{n}}{1-a_{n}+a_{n}\eta
\left( t_{n}\right) }\left \Vert x_{n}-p\right \Vert .  \label{7}
\end{equation}%
The same inequality holds trivially when $x_{n+1}=p$.

Define%
\begin{equation*}
\gamma _{n}:=\frac{a_{n}\eta \left( t_{n}\right) }{1-a_{n}+a_{n}\eta \left(
t_{n}\right) }.
\end{equation*}%
Then $0<\gamma _{n}<1$, and (\ref{7}) becomes%
\begin{equation*}
\left \Vert x_{n+1}-p\right \Vert \leq \left( 1-\gamma _{n}\right) \left \Vert
x_{n}-p\right \Vert .
\end{equation*}%
Since $0<\eta \left( t_{n}\right) \leq 1$ the denominator in the definition
of $\gamma _{n}$ is at most one. Therefore,%
\begin{equation*}
\gamma _{n}\geq a_{n}\eta \left( t_{n}\right) ,
\end{equation*}%
and (\ref{6}) implies $\sum \limits_{n}\gamma _{n}=\infty $. Lemma \ref%
{Lemma1}, applied with $u_{n}=\left \Vert x_{n}-p\right \Vert $ and $%
\varepsilon _{n}=0$, now gives $x_{n}\rightarrow p$ in norm.
\end{proof}

\begin{remark}
If $C$ contains a point $x\neq p$, strong continuity and (\ref{5})
necessarily imply $\eta \left( t\right) \rightarrow 0$ as $t\downarrow 0$.
Indeed,%
\begin{equation*}
\eta \left( t\right) \left \Vert x-p\right \Vert ^{2}\leq \left \langle
x-T\left( t\right) x,\text{ }j_{t,x,p}\right \rangle \leq \left \Vert
x-T\left( t\right) x\right \Vert \left \Vert x-p\right \Vert ,
\end{equation*}%
and consequently,%
\begin{equation*}
0<\eta \left( t\right) \leq \frac{\left \Vert x-T\left( t\right) x\right \Vert 
}{\left \Vert x-p\right \Vert }\rightarrow 0.
\end{equation*}%
Thus the time-dependent condition is consistent with strong continuity,
whereas a positive lower bound for $\eta \left( t\right) $ near zero would
force $C=\left \{ p\right \} $.
\end{remark}

\subsection{An example illustrating the time--weight interaction}

The following example verifies the assumptions of Theorem \ref{Theorem 2}
and illustrates the importance of coupling the iteration weights with the
semigroup times.

\begin{example}
\label{Example 1}Let $E$ be a nonzero real Banach space, let $C=E$, and fix $%
\lambda >0$. Define%
\begin{equation*}
T\left( t\right) x=e^{-\lambda t}x,\qquad t\geq 0,\qquad x\in E.
\end{equation*}%
Then $\mathcal{T}=\left \{ T\left( t\right) :t\geq 0\right \} $ is a strongly
continuous semigroup, since%
\begin{equation*}
T\left( 0\right) =I,\qquad T\left( s+t\right) =T\left( s\right) T\left(
t\right) ,
\end{equation*}%
and $t\mapsto T\left( t\right) x$ is norm-continuous for every $x\in E$.
Moreover,%
\begin{equation*}
\left \Vert T\left( t\right) x-T\left( t\right) y\right \Vert =e^{-\lambda
t}\left \Vert x-y\right \Vert ,
\end{equation*}%
so we may take $L\left( t\right) =e^{-\lambda t}$. Its common fixed point
set is%
\begin{equation*}
F\left( \mathcal{T}\right) =\left \{ 0\right \} .
\end{equation*}%
For every $x\in E$ and every $j\in J\left( x\right) $,%
\begin{equation*}
\left \langle T\left( t\right) x,j\right \rangle =e^{-\lambda t}\left \Vert
x\right \Vert ^{2}.
\end{equation*}%
Thus condition (\ref{5}) holds with%
\begin{equation*}
\eta \left( t\right) =1-e^{-\lambda t},\qquad t>0.
\end{equation*}%
In particular, $\eta \left( t\right) >0$ for $t>0$ and $\eta \left( t\right)
\rightarrow 0$ as $t\downarrow 0$.

Choose%
\begin{equation*}
a_{n}=t_{n}=\frac{1}{\sqrt{n+1}},\qquad n\geq 1.
\end{equation*}%
Then $a_{n}\in \left( 0,1\right) $, $t_{n}\rightarrow 0$, and%
\begin{equation*}
a_{n}L\left( t_{n}\right) =a_{n}e^{-\lambda t_{n}}<1.
\end{equation*}%
Since%
\begin{equation*}
\lim \limits_{t\downarrow 0}\frac{1-e^{-\lambda t}}{t}=\lambda ,
\end{equation*}%
we have%
\begin{equation*}
\lim \limits_{n\rightarrow \infty }\frac{a_{n}\eta \left( t_{n}\right) }{%
1/\left( n+1\right) }=\lambda .
\end{equation*}%
Consequently,%
\begin{equation*}
\sum \limits_{n=1}^{\infty }a_{n}\eta \left( t_{n}\right) =\infty .
\end{equation*}%
All assumptions of Theorem \ref{Theorem 2} are therefore satisfied, and the
implicit iteration converges strongly to zero.

The iteration can also be solved explicitly:%
\begin{equation*}
x_{n+1}=\frac{1-a_{n}}{1-a_{n}e^{-\lambda t_{n}}}x_{n}=\frac{1}{1+b_{n}}%
x_{n},\qquad b_{n}:=\frac{a_{n}}{1-a_{n}}\left( 1-e^{-\lambda t_{n}}\right) .
\end{equation*}%
Hence,%
\begin{equation*}
x_{n}=\left( \prod \limits_{k=1}^{n-1}\left( 1+b_{k}\right) ^{-1}\right)
x_{1}.
\end{equation*}%
To contrast this convergent choice with insufficient time controls, take
instead%
\begin{equation*}
a_{n}=t_{n}=\frac{1}{n+1}.
\end{equation*}%
Although $a_{n}\rightarrow 0$, $t_{n}\rightarrow 0$, and $\sum
\limits_{n}a_{n}=\infty $, we now obtain%
\begin{equation*}
0<b_{n}=\frac{1}{n}\left( 1-e^{-\lambda /\left( n+1\right) }\right) \leq 
\frac{\lambda }{n\left( n+1\right) }.
\end{equation*}%
Thus $\sum \limits_{n}b_{n}<\infty $. Since%
\begin{equation*}
0\leq \log \left( 1+b_{n}\right) \leq b_{n},
\end{equation*}%
the infinite product%
\begin{equation*}
c:=\prod \limits_{n=1}^{\infty }\left( 1+b_{n}\right) ^{-1}
\end{equation*}%
is strictly positive. Therefore,%
\begin{equation*}
x_{n}\rightarrow cx_{1},
\end{equation*}%
which is not a common fixed point when $x_{1}\neq 0$.

This calculation shows that the condition $\sum \limits_{n=1}^{\infty
}a_{n}=\infty $ alone does not ensure convergence to a common fixed point
when $t_{n}\rightarrow 0$.
\end{example}

\begin{remark}
The semigroup in Example \ref{Example 1} is contractive and therefore
nonexpansive. The example verifies the consistency of Theorem \ref{Theorem 2}
and illustrates its time-dependent control condition; it does not establish
an extension beyond the nonexpansive class.
\end{remark}

\section{Application to a nonlinear planar evolution equation}

We consider a nonlinear evolution equation whose solution semigroup
satisfies the directional contraction condition of Theorem \ref{Theorem 2}
but is not nonexpansive.

Let $E=%
\mathbb{R}
^{2}$, equipped with the Euclidean inner product, and let%
\begin{equation*}
C=\left \{ x\in 
\mathbb{R}
^{2}:\left \Vert x\right \Vert \leq 1\right \} .
\end{equation*}%
For $\theta \in 
\mathbb{R}
$, denote by $R_{\theta }$ the rotation through angle $\theta $, and define%
\begin{equation*}
Q=\left( 
\begin{array}{cc}
0 & -1 \\ 
1 & 0%
\end{array}%
\right) .
\end{equation*}%
Consider the initial value problem%
\begin{equation}
u^{\prime }\left( t\right) =-u\left( t\right) +16\left \Vert u\left(
t\right) \right \Vert Qu\left( t\right) ,\qquad u\left( 0\right) =x\in C.
\label{4.1}
\end{equation}%
Its solution is%
\begin{equation}
T\left( t\right) x=e^{-t}R_{16\left( 1-e^{-t}\right) \left \Vert x\right
\Vert }x,\qquad t\geq 0.  \label{4.2}
\end{equation}%
Indeed, $\left \Vert T\left( t\right) x\right \Vert =e^{-t}\left \Vert
x\right \Vert $, and differentiation of (\ref{4.2}) yields (\ref{4.1}). The
vector field is locally Lipschitz, and the displayed solution is defined for
all $t\geq 0$.

\textbf{Semigroup property and invariance. }Clearly\textbf{\ }$T\left(
0\right) =I$, and%
\begin{equation*}
\left \Vert T\left( t\right) x\right \Vert =e^{-t}\left \Vert x\right \Vert
\leq 1,
\end{equation*}%
so $T\left( t\right) C\subset C$. Writing $r=\left \Vert x\right \Vert $, we
obtain%
\begin{eqnarray*}
T\left( s\right) T\left( t\right) x &=&e^{-\left( s+t\right) }R_{16\left(
1-e^{-s}\right) e^{-t}r}R_{16\left( 1-e^{-t}\right) r}x \\
&=&e^{-\left( s+t\right) }R_{16\left( 1-e^{-\left( s+t\right) }\right) r}x \\
&=&T\left( s+t\right) x.
\end{eqnarray*}%
Formula (\ref{4.2}) also shows that every orbit is norm-continuous. Thus $%
\mathcal{T}=\left \{ T\left( t\right) :t\geq 0\right \} $ is a strongly
continuous semigroup on $C$. Its common fixed point set is%
\begin{equation*}
F\left( \mathcal{T}\right) =\left \{ 0\right \} .
\end{equation*}%
\textbf{Lipschitz bound. }Set $c_{t}=16\left( 1-e^{-t}\right) $. Using%
\begin{equation*}
\left \Vert \left( R_{\alpha }-R_{\beta }\right) y\right \Vert \leq \left
\vert \alpha -\beta \right \vert \left \Vert y\right \Vert ,
\end{equation*}%
we obtain, for $x,y\in C$,%
\begin{eqnarray*}
\left \Vert T\left( t\right) x-T\left( t\right) y\right \Vert &\leq
&e^{-t}\left( \left \Vert x-y\right \Vert +c_{t}\left \vert \left \Vert
x\right \Vert -\left \Vert y\right \Vert \right \vert \left \Vert y\right
\Vert \right) \\
&\leq &e^{-t}\left( 1+c_{t}\right) \left \Vert x-y\right \Vert .
\end{eqnarray*}%
Consequently, we may take%
\begin{equation*}
L\left( t\right) =e^{-t}\left( 1+16\left( 1-e^{-t}\right) \right) \leq 17.
\end{equation*}%
\textbf{Directional contraction. }Since rotations preserve norms,%
\begin{equation*}
\left \langle T\left( t\right) x,x\right \rangle \leq \left \Vert T\left(
t\right) x\right \Vert \left \Vert x\right \Vert =e^{-t}\left \Vert x\right
\Vert ^{2}.
\end{equation*}%
Therefore, condition (\ref{5}) holds with%
\begin{equation*}
\eta \left( t\right) =1-e^{-t}.
\end{equation*}%
\textbf{Failure of nonexpansiveness. }Let $t_{\ast }=\log 2$ and let $%
e_{1}=\left( 1,0\right) $. For $0<r<1$, define%
\begin{equation*}
g\left( r\right) :=T\left( t_{\ast }\right) \left( re_{1}\right) =\frac{r}{2}%
R_{8r}e_{1}.
\end{equation*}%
Differentiating gives%
\begin{equation*}
g^{\prime }\left( r\right) =\frac{1}{2}R_{8r}\left( e_{1}+8rQe_{1}\right) ,
\end{equation*}%
and hence%
\begin{equation*}
\left \Vert g^{\prime }\left( r\right) \right \Vert =\frac{1}{2}\sqrt{%
1+64r^{2}}.
\end{equation*}%
In particular,%
\begin{equation*}
\left \Vert g^{\prime }\left( \frac{1}{2}\right) \right \Vert =\frac{\sqrt{17%
}}{2}>1.
\end{equation*}%
Thus, for sufficiently small nonzero $h$,%
\begin{equation*}
\frac{\left \Vert T\left( t_{\ast }\right) \left( \left( \frac{1}{2}%
+h\right) e_{1}\right) -T\left( t_{\ast }\right) \left( \frac{1}{2}%
e_{1}\right) \right \Vert }{\left \vert h\right \vert }>1.
\end{equation*}%
Both initial points belong to $C$ when $h$ is sufficiently small. Therefore, 
$T\left( t_{\ast }\right) $ is not nonexpansive.

\textbf{Convergence of the implicit iteration. }Choose%
\begin{equation*}
a_{n}=\frac{1}{18\sqrt{n+1}},\qquad t_{n}=\frac{1}{\sqrt{n+1}},\qquad n\geq
1.
\end{equation*}%
Then%
\begin{equation*}
a_{n}L\left( t_{n}\right) \leq \frac{17}{18\sqrt{n+1}}<1.
\end{equation*}%
Moreover,%
\begin{equation*}
\lim \limits_{n\rightarrow \infty }\left( n+1\right) a_{n}\eta \left(
t_{n}\right) =\frac{1}{18},
\end{equation*}%
so%
\begin{equation*}
\sum \limits_{n=1}^{\infty }a_{n}\eta \left( t_{n}\right) =\infty .
\end{equation*}%
Theorem \ref{Theorem 2} therefore applies. For every $x_{1}\in C$, the
uniquely defined implicit sequence%
\begin{equation*}
x_{n+1}=\left( 1-a_{n}\right) x_{n}+a_{n}T\left( t_{n}\right) x_{n+1}
\end{equation*}%
converges strongly to zero, the unique equilibrium of (\ref{4.1}).

This application shows that Theorem \ref{Theorem 2} can apply to a nonlinear
semigroup containing mappings that are not nonexpansive. The conclusion
follows from contraction relative to the common equilibrium, rather than
from a nonexpansive estimate between arbitrary pairs of points.

\section{Discussion and scope of the results}

The weak and strong convergence results established above rely on different
assumptions. Theorem \ref{Theorem 1} combines Fej\'{e}r monotonicity with a
normalized residual condition, uniform continuity at zero along the
generated sequence, and demiclosedness at fixed positive times. Theorem \ref%
{Theorem 2} instead uses a quantitative directional contraction condition to
obtain norm convergence directly.

\textbf{Weak convergence and residual assumptions. }Theorem \ref{Theorem 1}
is a sufficient-condition result. Its normalized residual assumption and the
uniform continuity condition along the sequence must be verified separately
in any application. Neither condition follows from Fej\'{e}r monotonicity
alone. Moreover, strong continuity concerns each fixed initial point and
does not automatically imply uniform continuity at zero over a bounded
sequence of initial points. The theorem does not, by itself, provide control
conditions on $\left( a_{n}\right) $ and $\left( t_{n}\right) $ that
guarantee the required residual estimates.

\textbf{The role of the semigroup times. }When $t_{n}\rightarrow 0$, a
vanishing residual $\left \Vert T\left( t_{n}\right) x_{n}-x_{n}\right \Vert 
$ need not identify a common fixed point. Lemma \ref{Lemma4} addresses this
issue by imposing a normalized residual condition and an additional uniform
continuity assumption. Example \ref{Example 1} also shows that $%
a_{n}\rightarrow 0$, $\sum \limits_{n}a_{n}=\infty $, and $t_{n}\rightarrow
0 $ are insufficient to ensure convergence to a common fixed point, even for
a semigroup of linear contractions.

\textbf{Time-dependent directional contraction. }In Theorem \ref{Theorem 2},
the function $\eta \left( t\right) $ measures the directional contraction
relative to the common fixed point. On a nontrivial domain, strong
continuity requires $\eta \left( t\right) \rightarrow 0$ as $t\downarrow 0$.
The condition 
\begin{equation*}
\sum \limits_{n=1}^{\infty }a_{n}\eta \left( t_{n}\right) =\infty 
\end{equation*}%
ensures that the accumulated contraction remains sufficient for norm
convergence. If, for some $c>0$ and $\tau >0$,%
\begin{equation*}
\eta \left( t\right) \geq ct\qquad \left( 0<t\leq \tau \right) ,
\end{equation*}%
then $t_{n}\rightarrow 0$ and $\sum \limits_{n}a_{n}t_{n}=\infty $ imply this
condition. Such a lower bound on $\eta $, however, is an additional
quantitative assumption.

\textbf{Geometry and uniqueness. }Theorem \ref{Theorem 1} uses reflexivity
to obtain weakly convergent subsequences and Opial's condition to establish
uniqueness of their limits. Theorem \ref{Theorem 2} requires neither
assumption, but its directional contraction condition forces the common
fixed point to be unique. Thus, the strong convergence theorem concerns a
more restrictive dynamical setting and is not a strong convergence version
of Theorem \ref{Theorem 1} under the same hypotheses.

\textbf{Scope beyond nonexpansiveness. }The planar evolution semigroup in
Section 4 satisfies the assumptions of Theorem \ref{Theorem 2} while failing
to be nonexpansive at a positive time. This verifies that the theorem is
applicable to some semigroups outside the nonexpansive class. It does not
establish that the directional contraction condition contains all
nonexpansive or pseudocontractive semigroups, nor does it show that the
application lies outside every existing convergence framework.

\section{Conclusions}

We have studied an implicit iteration for approximating common fixed points
of strongly continuous semigroups in real Banach spaces. Under a stepwise
Lipschitz restriction, the iteration is well-defined and Fej\'{e}r monotone
with respect to the common fixed point set.

Two convergence results were established under distinct additional
assumptions. The weak convergence theorem uses a normalized residual
condition, uniform continuity at zero along the generated sequence, and
demiclosedness at fixed positive times in a reflexive Banach space
satisfying Opial's condition. These residual and continuity assumptions
require separate verification in applications. The strong convergence
theorem applies in general real Banach spaces under a time-dependent
directional contraction condition and a divergence condition coupling the
iteration weights with the contraction strength. This directional condition
also ensures uniqueness of the common fixed point.

A linear example demonstrates that vanishing time steps and a divergent sum
of iteration weights do not alone guarantee convergence to a common fixed
point. A nonlinear planar evolution equation provides an application of the
strong convergence theorem in which the solution semigroup is not
nonexpansive.

A remaining problem is to derive the residual assumptions of the weak
convergence theorem from independently verifiable conditions on the
semigroup and the control sequences. Another direction is to investigate
strong convergence when the common fixed point set is not a singleton, under
assumptions compatible with strong continuity.

\end{document}